\documentclass[a4paper,10pt]{article}

\usepackage[utf8x]{inputenc}
\usepackage[T1]{fontenc} 
\usepackage{amsmath,amsthm}
\usepackage[dvips]{graphicx}
\usepackage{amsfonts,amssymb}
\usepackage{geometry}
\usepackage{hyperref}
\usepackage{stmaryrd}
\usepackage{fancyhdr}
\usepackage{url}

\newtheorem{prop}{Proposition}[section]
\newtheorem{LM}{Lemma}[section]
\newtheorem{thm}{Theorem}[section]
\newtheorem{df}{Definition}[section]
\newtheorem{df-prop}{Definition-Proposition}[section]
\newtheorem{cor}{Corollary}[section]

\newtheorem*{thm*}{Theorem \ref{distgen}}
\newtheorem*{thm**}{Theorem \ref{eqasai}}
\newtheorem*{thme**}{Theorem \ref{asaiprinc}}

\newtheoremstyle{pourlesremarques}{\topsep}{\topsep}{\normalfont}{}{\bfseries}{.}{ }{}
\theoremstyle{pourlesremarques}
\newtheorem{rem}{Remark}[section]
\def\adots{\mathinner{\mkern2mu\raise 1pt\hbox{.}\mkern 3mu\raise
4pt\hbox{.}\mkern1mu\raise 7pt\hbox{{.}}}}

\title {Distinction and Asai $L$-functions for generic representations of general linear groups over p-adic fields}
\author{Nadir Matringe}
\date{}

\begin{document}

 \maketitle

\begin{abstract}
Let $K/F$ be a quadratic extension of $p$-adic fields, and $n$ a positive integer. A smooth irreducible representation of the group $GL(n,K)$ is said to be distinguished, 
if it admits on its space a nonzero $GL(n,F)$-invariant linear form. In the present work, we classify distinguished generic representations of the group $GL(n,K)$ in terms of inducing 
quasi-square-integrable representations. This has as a consequence the truth of the expected equality between the Rankin-Selberg type Asai $L$-function of a generic 
representation, and the Asai $L$-function of its Langlands parameter.
\end{abstract}

\section*{Introduction}

Given $K/F$ a quadratic extension of $p$-adic fields, we denote by $\sigma$ the non trivial element of the Galois group of $K$ over $F$. We denote by $\eta_{K/F}$ 
the character of order $2$ of $F^*$, trivial on the set of norms of $K^*$.\\

 A smooth representation of $GL(n,K)$ is said to be distinguished if it admits on its space a nonzero linear form, which is invariant under $GL(n,F)$. 
The pair $(GL(n,K),GL(n,F))$ is known to be a generalized Gelfand pair, which means that for an irreducible representation $(\pi,V_{\pi})$ of $GL(n,K)$, 
the space of $GL(n,F)$-invariant linear forms on $V_{\pi}$ is of dimension at most one. The unitary distinguished representations are the natural space which supports 
the Plancherel measure of the symmetric space $GL(n,F)\backslash GL(n,K)$. Hence their understanding is related to harmonic analysis on $GL(n,F)\backslash GL(n,K)$.\\

We classify here distinguished generic representations of $GL(n,K)$, in terms of inducing discrete series representations. More precisely we prove the following result.

\begin{thm*}
Let $\pi$ be a generic representation of the group $GL(n,K)$, obtained by normalized parabolic induction of quasi-square-integrable representations 
$\Delta_1, \dots ,\Delta_t$. It is distinguished if and only if there exists a reordering of the ${\Delta _i}$'s, and an integer $r$ between $1$ and $t/2$, such that we have 
$\Delta_{i+1}^{\sigma} = \Delta_i^{\vee} $ for $i=1,3,..,2r-1$, and $\Delta_{i}$ is distinguished for $i > 2r$.
\end{thm*}

Distinguished representations of $GL(n,K)$ are also related by a conjecture of Flicker and Rallis, to the base change theory of representations of a unitary group 
(see for example \cite{AR}). The main result of this paper could perhaps reduce the proof of this conjecture for generic distinguished representations of $GL(n,K)$, 
to the case of discrete series representations.\\

Generic distinguished representations are related to the analytic behaviour of meromorphic functions, called Asai $L$-functions associated with generic representations.\\
The basic theory of the Asai $L$-function of a generic representation $\pi$ of $GL(n,K)$, denoted by $L_F^K(\pi)$ and defined as the gcd of functions obtained as meromorphic 
extension of Rankin-Selberg integrals, such as its functional equation, has first been developed by Flicker in \cite{F1} and \cite{F3}.\\

Then in \cite{K}, Kable proves that if the Asai $L$-function $L_F^K(\pi)$ of a discrete series representation $\pi$ admits a pole at zero, then the representation $\pi$ is 
distinguished. This, with the equality of the product of the Asai $L$-functions of $\pi$ and $\eta \pi$ (for a character $\eta$ of $K^*$ extending $\eta_{K/F}$), and of the 
classical $L$-function of the pair $(\pi^{\sigma},\pi^{\vee})$, obtained by a global-local method, allows him to prove the so called Jacquet conjecture for discrete series 
representations. This result states that a discrete series representation $\pi$ of $GL(n,K)$ which is Galois autodual (i.e. $\pi^{\vee}=\pi^\sigma$), is either 
distinguished or $\eta_{K/F}$-distinguished.\\
This result is specified in \cite{AKT}, where it is shown that the preceding (either/or) is exclusive, by first proving that the Asai $L$-function of a tempered 
distinguished representation has a pole at zero.\\

For non discrete series representation, the correct statement is that of Theorem \ref{distgen}. Actually this theorem can be seen as a generalization of Jacquet's conjecture 
for discrete series representations, 
as it says that a generic representation of $GL(n,K)$ is Galois autodual if and only if it is parabolically induced from three representations, one that is distinguished
 but not $\eta_{K/F}$-distinguished, one that is distinguished and $\eta_{K/F}$-distinguished, and one that is distinguished but not $\eta_{K/F}$-distinguished. 
Among these, the distinguished are those with purely $\eta_{K/F}$-distinguished part equal to zero.\\
The last step before Theorem \ref{distgen}, consisting of showing that the representations described in the theorem are indeed distinguished is the main result of \cite{M4}.\\

Concerning Rankin-Selberg type Asai $L$-functions, a definitive statement relating their poles and distinction is obtained in \cite{M2}, where it is proved that a representation
 $\pi$ is distinguished, if and only if its Asai $L$-function admits a so called (in the terminology of \cite{CP}) exceptional pole at zero.\\ 

There are two other ways to associate an Asai $L$-function to a representation $\pi$ of the group $GL(n,K)$.\\
 The first is by considering the $n$-dimensional representation $\rho$ of the Weil-Deligne group $W'_K$ of $K$, associated to $\pi$ by the local Langlands correspondence. 
One then defines by multiplicative induction a representation of the Weil-Deligne group $W'_F$ (which contains $W'_K$ as a subgroup of index $2$), of dimension $n^2$, 
denoted by $M_{W'_K}^{W'_F}(\rho)$. The Asai $L$-function corresponding to $\rho$ is by definition the classical $L$-function of the 
representation $M_{W'_K}^{W'_F}(\rho)$, which we denote by $L_F^K(\rho)$.\\
 The second, called the Langlands-Shahidi method, is introduced in \cite{Sh}. We denote by $L_F^{K,U}(\pi)$ the meromorphic function obtained by this process, the study of 
its poles is this time related to the fact of knowing when a representation $\pi$ is obtained by base change lift from a unitary group (see \cite{Go}).
It is conjectured that these three functions are actually the same (cf. \cite{He}, \cite{K}, \cite{AR}). Henniart proved in \cite{He} that the functions $L_F^{K,U}(\pi)$ and 
$L_F^{K}(\rho)$ are equal for every irreducible representation $\pi$. 
Anandavardhanan and Rajan proved in \cite{AR} that the functions $L_F^K(\pi)$ and $L_F^{K,U}(\pi)$ are equal for $\pi$ in the discrete series of $GL(n,K)$.\\

In \cite{M3}, which can be used as a survey for local Rankin-Selberg type Asai $L$-functions, Theorem \ref{distgen} is stated as a conjecture. It is then showed using a 
method of Cogdell and Piatetski-Shapiro and the known equality of $L_F^{K}(\pi)$ and $L_F^K(\rho)$ for discrete series representations, that the theorem implies the equality of 
$L_F^{K}(\pi)$ and $L_F^K(\rho)$ for generic representations.\\
Hence we have the following result.

\begin{thm**}
Let $\pi$ be a generic representation of the group $GL(n,K)$, and let $\rho$ be the representation of dimension $n$ of the Weil-Deligne group $W'_K$ of $K$, 
corresponding to $\pi$ through Langlands correspondence. The following equality of $L$-functions is satisfied:
$$ L_F^K(\pi,s) = L_F^K(\rho,s) .$$ 
\end{thm**}

Now the main tool for the proof of Theorem \ref{distgen}, classical in this situation, is Mackey theory. For instance the case of principal series representations 
is treated in section 6 of \cite{JLR}.\\ 
 A generic representation $\pi$ of the group $GL(n,K)$, is obtained by normalized parabolic induction of a discrete series representation 
$\Delta=\Delta_1\otimes\dots\otimes\Delta_t$ of a standard Levi subgroup of a parabolic subgroup $P(K)$ of $GL(n,K)$. 
Calling $R$ a set of representatives of the double classes 
$P(K)\backslash GL(n,K)/GL(n,F)$, the restriction to $GL(n,F)$ of the representation $\pi$ has a factor series, with factors being induced representations of 
$\Delta_{|P(K)\cap uGL(n,F)u^{-1}}$ to the group $uGL(n,F)u^{-1}$, for some $u$ in $R$. If the representation $\pi$ is $GL(n,F)$-distinguished, then it is at least the case for one of the factors. 
But using Frobenius reciprocity law, such a factor is distinguished, if and only if the representation $\Delta$ is itself $\chi$-distinguished by 
$P(K)\cap uGL(n,F)u^{-1}$ fore some character $\chi$. 
Studying the structure of this subgroup, implies that the $\Delta_i$'s must be of the requested form. The group structure of $P(K)\cap uGL(n,F)u^{-1}$ 
was studied in \cite{JLR} when the element $uu^{-\sigma}$ (which is actually an involution in $\mathfrak{S}_n$) normalizes the standard Levi subgroup of $P$. 
Here we reduce the general case to this particular case.\\

The first part of Section 1 is about definitions and notations of the basic objects we use.\\
The second part concerns itself with results of Bernstein and Zelevinsky about classification of discrete series representations in terms of segments, and the computation of 
their Jacquet modules.\\
A set of representatives $R$ of the double classes $P(K)\backslash GL(n,K)/GL(n,F)$ for a standard parabolic subgroup 
$P(K)$ of $GL(n,K)$ was already described in \cite{JLR}. In Section 2, we give a geometric way of obtaining this set, which has its advantage. 
We then show how to reduce the study of the group structure of $P(K)\cap uGL(n,F)u^{-1}$ to the admissible case treated in \cite{JLR}.
 We end by computing the modulus character of this group.\\ 
We eventually prove Theorem \ref{distgen} in the last section.

\subsection*{Acknowledgements}
As this paper is a conclusion of my PHD work, I would like to thank people who helped me during the process.
I thank Corinne Blondel, my advisor Paul Gérardin and my reviewers U.K. Anandavardhanan and Anthony Kable.
 I thank Herv\'e Jacquet, who gave me very helpful bibliographic information about the subject, and for sending me notes containing a statement equivalent to Theorem \ref{distgen} as a conjecture, around the time it was also conjectured in \cite{M3}. I thank Paul G\'erardin, Ivan Marin and J-F. Planchat for fruitful discussions during the writing of the paper. 

\section{Preliminaries}

\subsection{Notations and definitions}

We fix until the end of this paper a local field $F$ of characteristic zero. We fix a quadratic extension $K$ of $F$.\\

If $G$ is a group acting on two vector spaces $V$ and $V'$, then $Hom_G(V,V')$ designates the space of $G$-equivariant morphisms from $V$ to $V'$. \\
If $E$ is a finite extension of $F$, we denote by $v_E$ the discrete valuation of $E$, which verifies that $v_E(\pi_E)$ is $1$ if $\pi_E$ is a prime element of $E$.
 We denote by $q_E$ the cardinality of the residual field of $E$. We denote by $|\ |_E$ the absolute value of $E$ defined by $|x|_E = q_E^{-v_E (x)}$, for $x$ in $E$.
 We denote by $R_E$ the valuation ring of $E$, and by $P_E$ the maximal ideal of $R_E$. Finally we denote by $W_E$ the Weil group of $E$ (cf. \cite{T}), and by $W'_E$
 the Weil-Deligne group of $E$. The group $W'_E$ is the semi-direct product $W_E \rtimes SL(2,\mathbb{C})$, with $W_E$ acting by its quotient group 
$q_E^\mathbb{Z}$ on $SL(2,\mathbb{C})$. A Frobenius element in $W_E$ acts on $SL(2,\mathbb{C})$ is by conjugation by 
the matrix $\left( \begin{array}{cc}
q_E & 0\\
0 & 1 \end{array}\right)$.\\

Let $G$ be an affine algebraic group defined over the field $F$. If $E$ is an extension of $F$, we denote by $G(E)$ the group of the $E$-points of $G$. 
Such a group is locally compact and totally disconnected, we will call it an $l$-group. \\
Let $n$ be a positive integer, we denote by $M_n= M_n(\bar{F})$ the additive group of $n\times n$ matrices with entries in $\bar{F}$, 
and we denote by $G_n$ the general linear group $GL(n,\bar{F})$ of invertible matrices of $M_n(\bar{F})$. If $M$ belongs to $M_n$, we denote its determinant by $det(M)$.\\
We call partition of a positive integer $n$, a family $\bar{n}=(n_1,\dots,n_t)$ of positive integers (for a certain $t$ in $\mathbb{N}-\left\lbrace 0\right\rbrace $), 
such that the sum $n_1+\dots+n_t$ is equal to $n$. To such a partition, we associate an algebraic subgroup of $G_n$ denoted by $P_{\bar{n}}$, 
given by matrices of the form $$\left (\begin{array}{cccccccc}
g_1 & \star & \star & \star & \star \\
    & g_2 & \star &  \star & \star \\
    &  & \ddots & \star & \star \\  
    &  &     & g_{t-1} &  \star \\
    &  &     &   & g_t
 \end{array}\right),$$ with $g_i$ in $G_{n_i}$ for $i$ between $1$ and $t$. We call it the standard parabolic subgroup associated with the partition $\bar{n}$.
 We call parabolic subgroup any conjugate of a standard parabolic subgroup. We denote by $N_{\bar{n}}$ its unipotent radical subgroup, 
given by the matrices $$\left (\begin{array}{cccccccc}
I_{n_1} & \star & \star\\
     &  \ddots & \star\\  
    &  & I_{n_t}
 \end{array}\right)$$ and by $M_{\bar{n}}$ its standard Levi subgroup given by the matrices $$\left (\begin{array}{cccccccc}
g_{1} &   &  \\
     & \ddots  &  \\  
    &  & g_{t}
 \end{array}\right) $$ with the $g_i$'s in $G_{n_i}$. The group $P_{\bar{n}}$ identifies with the semidirect product $N_{\bar{n}}\rtimes M_{\bar{n}}$.\\

For $\bar{n}=(1,\dots,1)$, we denote by $B_n$ the Borel subgroup $P_{\bar{n}}$, and by $T_n$ its standard Levi subgroup. The Lie algebra of $Lie(G_n)$ decomposes as
$Lie(T_n)\oplus ( \oplus_{\alpha \in \Phi} Lie(N_{\alpha}))$ under the adjoint action of $T_n$. The group $T_n$ acts on each one-dimensional space $Lie(N_{\alpha})$ 
by a non-trivial 
character $\alpha$, and by the trivial character on $Lie(T_n)$. The set $\Phi$ is the root system of $G_n$ with respect to $T_n$. The group $N_{\alpha}$ is the unipotent 
subgroup of $G_n$ with Lie algebra $Lie(N_{\alpha})$. We denote by $\Phi^+$ the roots $\alpha$ such that $N_{\alpha}\subset B_n$ and by $\Phi^-$ the other roots. 
We denote by $\Phi_M$ or $\Phi_{\bar{n}}$ the set of roots $\alpha$ with $N_\alpha \subset M$ for $M=M_{\bar{n}}$, by $\Phi_M^+$ and $\Phi_M^-$ the intersections 
$\Phi_M \cap \Phi^+$ and $\Phi_M \cap \Phi^-$ respectively.\\
 The quotient $W=N_{G_n}(T_n)/T_n$ of the normalizer of $T_n$ in $G_n$ by $T_n$, is called the Weyl group of $G_n$, it identifies with the symmetric group 
$\mathfrak{S_n}$, and permutes the roots in $\Phi$. The quotient $W_M=N_{M_{\bar{n}}}(T_n)/T_n$ of the normalizer of $T_n$ in $M_{\bar{n}}$ by $T_n$, 
is called the Weyl group of $M_{\bar{n}}$, it identifies with the product group $\prod_{i=1}^t \mathfrak{S_{n_i}}$, and permutes the roots in $\Phi_M$. \\

Let $G$ be an $l$-group (i.e. locally compact totally disconnected), we denote by $d_G g$ or simply $dg$ if the context is clear, a left Haar measure $G$. For $x$ in $G$, we denote by $\delta_G (x)$ the positive number defined by the relation $d_g (gx)= \delta_G^{-1} (x)d_g (g)$. The modulus character $\Delta_G$ defines a morphism from $G$ into $\mathbb{R}_{>0}$. We denote by $\delta_G$ (which we also call modulus character) the morphism from $G$ into $\mathbb{R}_{>0}$ defined by $x \mapsto \Delta_G(x^{-1})$.\\

Let $G$ be an $l$-group, and $H$ a subgroup of $G$, a representation $(\pi,V)$ of $G$ is said to be smooth 
if for any vector $v$ of the vector space $V$, there is an open subgroup $U_v$ of $G$ stabilizing $v$ through $\pi$. We denote by $V^H$ the subspace of fixed points of
 $V$ under $H$. The category of smooth representations of $G$ is denoted by $Alg(G)$. If $(\pi,V)$ is a smooth representation of $G$, we denote by $\pi^{\vee}$ its dual 
representation in the smooth dual space $\tilde{V}$ of $V$.\\
 We will only consider smooth complex representations of $l$-groups.\\

\begin{df}
Let $G$ be an $l$-group, $H$ a closed subgroup of $G$, and $(\pi,V)$ a representation of $G$. If $\chi$ is a character of $H$, we say that the representation $\pi$ is $\chi$-distinguished under $H$, if it admits on its space a nonzero linear form $L$, verifying $L(\pi(h)v)=\chi(h)L(v)$ for all $v$ in $V$ and $h$ in $H$. If $\chi=1$, we say $H$-distinguished instead of $1$-distinguished. We omit ``$H$-'' if the context is clear.
\end{df}

If $\rho$ is a complex representation of $H$ in  $V_{\rho}$, we denote by $C^{\infty}(H \backslash X, \rho, V_{\rho})$ the space of functions $f$ from $X$ to $V_{\rho}$, 
fixed under the action by right translation of some compact open subgroup $U_f$ of $G$, and which verify $f(hx)=\rho(h)
 f(x)$ for $h \in H$, and $x \in X$. We denote by $C_c^{\infty}(H \backslash X, \rho, V_{\rho})$ the subspace of functions with support compact modulo 
$H$ of $C^{\infty}(H \backslash X, \rho, V_{\rho})$.\\ 
We denote by $Ind _H^G (\rho)$ the representation by right translation of $G$ in $C^{\infty}(H \backslash G, \rho,
 V_{\rho})$ and by $ind _H^G (\rho)$ the representation by right translation of $G$ in $C_c^{\infty}(H \backslash G,  \rho,
 V_{\rho})$.  We denote by ${Ind'} _H^G (\rho)$ the normalized induced representation $Ind _H^G ((\Delta _G /\Delta _H)^{1/2} \rho)$ and by ${ind'} _H^G (\rho)$ 
the normalized induced representation $ind _H^G ((\Delta _G /\Delta _H)^{1/2} \rho)$.\\  
Let $n$ be a positive integer, and $\bar{n}=(n_1,\dots,n_t)$ be a partition of $n$, and suppose that we have a representation $(\rho_i, V_i)$ of $G_{n_i}(K)$ for each 
$i$ between $1$ and $t$. 
Let $\rho$ be the extension to $P_{\bar{n}}$ of the natural representation 
$\rho_1 \otimes \dots \otimes \rho_t$ of $M_{\bar{n}}\simeq G_{n_1}(K) \times \dots \times G_{n_t}(K)$, trivial on $N_{\bar{n}}$. 
We denote by $\rho_1 \times \dots \times \rho_t$ the representation ${ind'} _{P_{\bar{n}}(K)}^{G_n (K)} (\rho)= {Ind'} _{P_{\bar{n}}(K)}^{G_n (K)} (\rho)$.

\subsection{Segments and quasi-square integrable representations}

From now on we assimilate representations to their isomorphism classes.\\

In this subsection, we recall results of Bernstein and Zelevinsky about quasi-square-integrable representations, more precisely their classification in terms of segments
 associated to supercuspidal representations, and how to compute their Jacquet modules.\\

If $\pi$ is an irreducible representation of $G_n(K)$, one denotes by $c_{\pi}$ its central character.\\
We recall that an irreducible representation of $G_n(K)$ is called supercuspidal if all its Jacquet modules associated to proper standard Levi subgroups are zero, which is 
equivalent to the fact that it has a coefficient with support compact modulo the center $Z_n(K)$ of the group $G_n(K)$.\\

An irreducible representation $\pi$ of the group $G_n(K)$ is called quasi-square-integrable, if there exists a positive character $\chi$ of the multiplicative group $K^*$, 
such that one of the coefficients $g\mapsto c(g)$ of $\pi$ verifies that $c(g)\chi(det(g))$ is a square-integrable function for a Haar measure of $G_n(K)/Z_n(K)$. One says that the representation $\pi$ is  square-integrable (or belongs to the discrete series of $G_n(K)$) if one can choose $\chi$ to be trivial.\\
If $\rho$ is a supercuspidal representation of $G_r(K)$ for a positive integer $r$, one denotes by $\rho|\ |_K$ the representation obtained by twist with the character 
$|\det(\ )|_K$.\\
 In general, if $E$ is an extension of $F$, and $\chi$ is a character of $E^*$, we will still denote by $\chi$ the character $\chi\circ det$.\\
 We call segment a list $\Delta$ of supercuspidal representations of the form $$\Delta=[\rho|\ |_K^{l-1},\rho|\ |_K^{l-2},\dots, \rho] $$ for a positive integer $l$. 
We call length of the segment the integer $rl$.  
We have the following theorem (Theorem 9.3 of \cite{Z}) that classifies quasi-square integrable representations in terms of segments.

\begin{thm} Let $\rho$ be a supercuspidal representation of $G_r(K)$ for a positive integer $r$.
 The representation $\rho \times \rho|\ |_F \times \dots \times \rho|\ |_F^{l-1}$ of $G_{rl}(K)$ is reducible, with a unique irreducible quotient 
that we denote by $[\rho|\ |_K^{l-1},\rho|\ |_K^{l-2},\dots, \rho]$. A representation $\Delta$ of the group $G_n(K)$ is quasi-square-integrable if and only if 
there is $r\in \left\lbrace 1,\dots, n \right\rbrace$ and $l\in \left\lbrace 1,\dots, n \right\rbrace$ with $lr=n$, and $\rho$ a supercuspidal representation
 of $G_r(K)$, such that the representation $\Delta$ is equal to $[\rho|\ |_K^{l-1},\rho|\ |_K^{l-2},\dots, \rho]$. The representation $\rho$ is then unique. \end{thm}

A representation of this type is  square-integrable if and only if it is unitarizable, or equivalently if and only if $\rho|\ |_F^{(l-1)/2}$ is unitarizable 
(i.e. its central character is unitary). We say that two segments are linked if none of them is a subsegment of the other, but their union is still a segment.\\
Now we allow ourselves to call segment a quasi-square-integrable representation, and to denote such a representation by its corresponding segment.\\
We will also use the following useful notation: if $\Delta_1$ and $\Delta_2$ are two disjoint segments, which are linked, and such that $\Delta_1$ precedes $\Delta_2$
 (i.e. the segment $\Delta_1$ is of the form $[\rho_1|\ |_K^{l_1-1},\rho_1|\ |_K^{l_1-2},\dots, \rho_1]$, the segment $\Delta_2$ is of the form 
$[\rho_2|\ |_K^{l_2-1},\rho_1|\ |_K^{l_2-2},\dots, \rho_2]$, with $\rho_1=\rho_2|\ |_K^{l_2}$), we denote by $[\Delta_1,\Delta_2]$ the segment 
$[\rho_1|\ |_K^{l_1-1},\dots,\rho_2]$.

Let $P$ be a standard parabolic subgroup of $G_n(K)$, $M$ its standard Levi subgroup, and let $P'$ be a standard parabolic subgroup of $M$, with standard Levi subgroup $M'$,
 and unipotent radical $N'$. We recall that the normalized Jacquet module of a representation $(\rho,V)$ of $M$, associated to $M'$, which we denote by $r_{M',M}(\rho)$, is the representation of $M'$ on the space $V/V(N')$ (where $V(N')$ is the subspace of $V$ generated
by vectors of the form $v-\pi(n')v$ for $v$ in $V$ and $n'$ in $N'$), defined by $r_{M',M}(\rho)(m')(v+V(N'))= \delta_{M'}^{-1/2}\rho(m')v +V(N')$.

The following proposition (Proposition 9.5 of \cite{Z}), explains how to compute normalized Jacquet modules of segments.

\begin{prop}\label{Z}
Let $\rho$ a supercuspidal representation of $G_r(K)$ for a positive integer $r$. Let $\Delta$ be the segment $[\rho|\ |_K^{l-1},\rho|\ |_K^{l-2},\dots, \rho]$, for a positive integer $l$. Let $M$ be a standard Levi subgroup of $G_{lr}(K)$ associated with a partition $(n_1,\dots,n_t)$ of $lr$.\\
 The representation $r_{M,G}(\Delta)$ is zero, unless $(n_1,\dots,n_t)$ admits $\underbrace{(r,\dots,r)}_{l \ times}$ as a sub partition, in which case $\Delta$ is of the form $[\Delta_1,\dots,\Delta_t]$, with $\Delta_i$ of length $n_i$, and $r_{M,G}(\Delta)$ is equal to the tensor product $\Delta_1\otimes\dots\otimes\Delta_t$.  

\end{prop}

\section{Double classes $P(K)\backslash G_n(K)/G_n(F)$}

Let $\bar{n}$ be a partition $(n_1 ,\dots,n_{t})$ of a positive integer $n$, we denote by $P$ the standard parabolic subgroup $P_{\bar{n}}(K)$ of $G=G_n(K)$, 
by $M$ its standard Levi subgroup, by $B$ the group $B_n(K)$, by $T$ the group $T_n(K)$. By abuse of notation, the group $N_\alpha$ will be $N_{\alpha}(K)$. 
We denote by $H$ the group $G_n(F)$.\\
We study in first place the double classes of $H\backslash G/P$. This has already been done in Section 6 of \cite{JLR} and one can check that the representatives we obtain 
are the same they obtain (see Remark \ref{JLR}). This set is described in \cite{JLR} as the involutions in the set of left and right $W_M$-reduced elements of $W$.
 However we give a more geometric proof in which every representative $u$ of the double classes of $H\backslash G/P$ comes naturally equipped with a 
sub-partition $s$ of $(n_1 ,\dots,n_{t})$, or equivalently a standard parabolic subgroup $P_s$ of $G$ contained in $P$. This new standard parabolic subgroup will have 
the nice property that the representative $u$ is $P_s$-admissible (in the terminology of \cite{JLR}, see Definition 1 of section 6).
 This will then allow us to reduce to the study of the group $P\cap uH u^{-1}$ to that of the group $P_s\cap uH u^{-1}$, which has been carried out in \cite{JLR}.\\

We identify the quotient space $G/P$ with a flag manifold given by sequences (called $\bar{n}$-flags) $0\subset V_1 \subset ... \subset V_{t-1} \subset V = K^n$, 
where $V_j$ is a vector subspace of $V$, of dimension $n_1+ \dots + n_j$. Studying the double classes of $H\backslash G/P$, is then equivalent to understand the 
$H$-orbits of the flag manifold $G/P$. This is done in the following theorem.

\begin{thm}\label{H-orbits} The $H$-orbits of the flag manifold $G/P$, are characterized by the integers $dim(V_i \cap V_j^{\sigma})$, for $1\leq i \leq j \leq t-1$, 
which means that two $\bar{n}$-flags $D=0\subset V_1 \subset ... \subset V_{t-1} \subset V$ and $D'=0\subset V'_1 \subset ... \subset V'_{t-1} \subset V$ are in the same orbit 

under $H$, if and only if $dim(V_i \cap V_j^{\sigma})=dim(V'_i \cap {V'}_j^{\sigma})$ for $1\leq i \leq j \leq {t-1}$.

\end{thm}
 
\begin{proof}
We first state the following classical lemma.

\begin{LM}\label{F-base}
Let $V=K^n$, and $V_F=F^n \subset V$, the $F$-subspace of vectors of $V$ fixed by $\sigma$. A vector subspace $V'$ of $V$ verifies that $V'=V'^{\sigma}$, if and only if one can 
choose a basis of $V'$ in $V_F$, in which case one says that $V'$ is defined over $F$. Any subspace defined over $F$, has a supplementary subspace defined over $F$.
\end{LM}

Now we prove a second lemma about the filtration of $V$ in terms of $V_i\cap V_j^\sigma$ for $V_i$ and $V_j$ in the set of subspaces defining a $\bar{n}$-flag.\\
Let $D$ be a $\bar{n}$-flag, given by the sequence $D=0\subset V_1 \subset ... \subset V_{t-1} \subset V$. We set $V_0=0$ and $V_t=V$.\\
 For $1\leq i \leq j \leq t$, we denote by $S_{i,j}$ a supplementary space of $V_i \cap V_{j-1}^{\sigma}+ V_{i-1} \cap V_j^{\sigma}$ in $V_i \cap V_j^{\sigma}$.\\
 If $i=j$, we add the condition that the supplementary space $S_{i,i}$ we choose is defined over $F$, which is possible according to Lemma \ref{F-base}.\\
 Eventually, for $1\leq i \leq j \leq t$, we denote by $S_{j,i}$, the space $S_{i,j}^{\sigma}$, which is a supplementary space of 
$V_j \cap V_{i-1}^{\sigma} + V_{j-1} \cap V_i^{\sigma}$ in $V_j \cap V_i^{\sigma}$ .\\ 

\begin{LM}\label{dec}
 With these notations, if $(i,j)$ belongs to $\left\lbrace 1,\dots, t \right\rbrace^2$, the space $V_{i-1} +V_i\cap V_j^\sigma$ is equal to the sum 
$$(S_{1,1}\oplus \ldots \oplus S_{1,t})\oplus \ldots \oplus (S_{i-1,1}\oplus \ldots \oplus S_{i-1,t})\oplus (S_{i,1}\oplus \ldots \oplus S_{i,j}).$$
In particular, the space $V_i$ is equal to the direct sum $$(S_{1,1}\oplus \ldots \oplus S_{1,t})\oplus \ldots \oplus (S_{i,1}\oplus \ldots \oplus S_{i,t}).$$
\end{LM}
\begin{proof} Let $x$ belong to $V_{i-1} +V_i\cap V_j^\sigma$, then one can write $x= x_{i-1}+ y_{i,j}$ with $x_{i-1}$ in $V_{i-1}$ and $y_{i,j}$ in $V_i\cap V_j^\sigma$.
 But then $y_{i,j}= y_{i-1,j}+ y_{i,j-1}+s_{i,j}$ with $y_{i-1,j}$ in $V_{i-1}\cap {V_j}^\sigma$, $y_{i,j-1}$ in $V_{i}\cap {V_{j-1}}^\sigma$ and $s_{i,j}$ in $S_{i,j}$. 
Hence we have $x= x'_{i-1}+ y_{i,j-1} +s_{i,j}$, with $x'_{i-1}=x_{i-1}+ y_{i-1,j}$ belonging to $V_{i-1}$. So $x$ belongs to $(V_{i-1} +V_i\cap {V_{j-1}^\sigma}) + S_{i,j}$.
 But by definition of $S_{i,j}$, the preceding sum is actually direct, i.e. $x$ belongs to $(V_{i-1} +V_i\cap V_{j-1}^{\sigma}) \oplus S_{i,j}$. We thus proved that 
$V_{i-1} +V_i\cap {V_j}^\sigma= (V_{i-1} +V_i\cap {V_{j-1}^\sigma}) \oplus S_{i,j}$, and the proof ends by induction.  

\end{proof}

 Getting back to the proof of Theorem \ref{H-orbits}, it is obvious that if two $\bar{n}$-flags $D$ and $D'$ are in the same $H$-orbit, then one must have 
$dim(V_i \cap V_j^{\sigma})=dim(V'_i \cap {V'}_j^{\sigma})$ for $1\leq i \leq j \leq t$.\\
 Conversely, suppose that two $\bar{n}$-flags $D$ and $D'$, satisfy the condition $dim(V_i \cap V_j^{\sigma})=dim(V'_i \cap {V'}_j^{\sigma})$ for $1\leq i \leq j \leq t$.\\
 The assumption $dim(V_i \cap V_j^{\sigma})=dim(V'_i \cap {V'}_j^{\sigma})$ for $1\leq i \leq j \leq t$, implies that for any couple 
$(i,j) \in \left\lbrace 1,\dots, t \right\rbrace^2$, $S_{i,j}$ and $S'_{i,j}$ have the same dimension. For $1\leq i < j \leq t$, we choose a $K$-linear isomorphism 
$h_{i,j}$ between $S_{i,j}$ and $S'_{i,j}$. This defines an isomorphism $h_{j,i}$ between $S_{j,i}$ and $S'_{j,i}$, verifying $h_{j,i}(v)= (h_{i,j}(v^\sigma))^\sigma$ for all 
$v$ in $S_{j,i}$.\\
 Eventually, for each $l$ between $1$ and $t$, as $S_{l,l}$ and $S'_{l,l}$ are defined over $F$, we choose an isomorphism $h_{l,l}$ between $S_{l,l}$ and $S'_{l,l}$, 
verifying that $h_{l,l}(v^\sigma)=h_{l,l}(v)^\sigma$ for all $v\in S_{l,l}$.\\
As the space $V$ is equal to the sum $\underset{(k,l)\in \left\lbrace 1,\dots, t \right\rbrace^2}{\oplus}S_{k,l}$, and $V'$ is equal to 
$\underset{(k,l)\in \left\lbrace 1,\dots, t \right\rbrace^2}{\oplus}S'_{k,l}$, the $K$-linear isomorphism
 $h=\underset{(k,l)\in \left\lbrace 1,\dots, t \right\rbrace^2}{\oplus} h_{l,k}$ defines an element of $H$, sending $D$ to $D'$, so that $D$ and $D'$ are in the same $H$-orbit.
 \end{proof}

The proof of the previous theorem has as a consequence the following corollary.

\begin{cor}\label{classnumber}
The quotient $H\backslash G/P$ is a finite set, and its cardinality is equal to the number of sequences of positive or null integers $(n_{i,j})_{1\leq i \leq j \leq t}$, such that if we let $n_{j,i}$ be equal to $n_{j,i}$, then for $i$ between $1$ and $t$, one has $n_i = \sum_{j=1}^{t} n_{i,j}$. 
\end{cor}

\begin{df}
We call $I(\bar{n})$ the set of sequences described in the preceding corollary. 
\end{df}

Now to such a sequence, we are going to associate an element of $G$, which will be a representative of the corresponding double coset of $H\backslash G/P$. 
This will thus achieve the description of the set $H\backslash G/P$.\\
 First we recall that we denote by $V$ the space $K^n$, and that $P$ corresponds to a partition $\bar{n}=(n_1,\dots, n_t)$ of $n$. 
We denote by $B^0=(e_1,\dots,e_n)$ the canonical basis of $V$, and by $D^0$ the canonical $\bar{n}$-flag defined over $F$, given by 
$0\subset V^0_1\subset V^0_2 \subset V^0_{t-1} \subset V$, with $V^0_i= Vect(e_1,\dots,e_{(n_1+\dots n_{i})})$, corresponding to the sequence 
$n_{i,j}=0$ if $i<j$ and $n_{i,i}=n_i$.

\begin{prop} \label{HP}{(Representatives for $H\backslash G/P$)}
Let $(n_{i,j})_{1\leq i \leq j \leq t}$ be an element of $I(\bar{n})$. We denote by $V^0_{i,j}$ the space 
$$Vect(e_{(n_1 +\dots +n_{i-1}+ n_{i,1}+\dots + n_{i,j-1} + 1)},\dots,e_{(n_1 +\dots +n_{i-1}+ n_{i,1}+\dots +n_{i,j-1} + n_{i,j})}) ,$$ and we denote by 
$B^0_{i,j}$ its canonical basis
 $$\left\lbrace e_{(n_1 +\dots +n_{i-1}+ n_{i,1}+\dots + n_{i,j-1} + 1)},\dots,e_{(n_1 +\dots +n_{i-1}+ n_{i,1}+\dots +n_{i,j-1} + n_{i,j})}\right\rbrace.$$
  Hence one has $$V= (V^0_{1,1}\oplus \dots \oplus V^0_{1,t})\oplus (V^0_{2,1} \oplus \dots \oplus V^0_{2,t}) \oplus \dots \oplus (V^0_{t,1}\oplus \dots \oplus V^0_{t,t}).$$ 
We denote by $u'$ the element of $G$ sending $V^0_{i,i}$ onto itself, and $V^0_{i,j}\oplus V^0_{j,i}$ onto itself for $i < j$, whose restriction to 
$V^0_{i,i}$ has matrix $I_{n_{i,i}}$ in the basis $B^0_{i,i}$, and whose restriction to  $V^0_{\{ i,j \}}=V^0_{i,j}\oplus V^0_{j,i}$ has matrix 
$$\left( \begin{array}{cc} \frac{1}{2} I_{n_{i,j}} & \frac{1}{2} I_{n_{i,j}}\\
 -\frac{1}{2\delta} I_{n_{i,j}} & \frac{1}{2\delta} I_{n_{i,j}} \end{array} \right)$$ in the basis $B^0_{\{ i,j \}}= B^0_{i,j} \cup B^0_{j,i}$.\\
 The element $u'$ is a representative of the double coset of $H\backslash G/P$ associated with $(n_{i,j})_{1\leq i \leq j \leq t}$ in $I(\bar{n})$.
\end{prop}

\begin{proof} If $B_1=(v_i)$ and $B_2=(w_i)$ are two families of vectors of same finite cardinality in $V$, we denote by $\lambda B_1 + \mu B_2$ the family 
$(\lambda v_i + \mu w_i)$ for $\lambda$ and $\mu$ in $K$. With these notations, the element $u'$ sends $V^0_{i,j}$ onto $S_{i,j}=Vect(B^0_{i,j} -\frac{1}{2\delta} B^0_{j,i})$,
 and $V^0_{j,i}$ onto $S_{j,i}=S_{i,j}^\sigma$. Denoting $V^0_{i,i}$ by $S_{i,i}$ and one verifies from our choices that for $1\leq i\leq t$, one has 
$V_i=u(V^0_i)=(S_{1,1}\oplus \ldots \oplus S_{1,t})\oplus \ldots \oplus (S_{i,1}\oplus \ldots \oplus S_{i,t})$, and that $S_{i,j}$ is a supplementary space of 
$V_i \cap V_{j-1}^{\sigma}+ V_{i-1} \cap V_j^{\sigma}$ in $V_i \cap V_j^{\sigma}$. Hence the $\bar{n}$-flag $D$ corresponds to the sequence 
$(n_{i,j})_{1\leq i \leq j \leq t}$ of $I(\bar{n})$. 
 \end{proof}

A reformulation of what precedes is the following.

\begin{prop}\label{PH}{(Representatives for $P\backslash G/H$)}
A set of representatives of $P\backslash G/H$ is given by the elements $u=u'^{-1}$, where the $u'$ are as in Proposition \ref{HP}. The representative of the class associated 
with the sequence $(n_{i,j})_{1\leq i \leq j \leq t}$ in $I(\bar{n})$, restricts to $V^0_{i,i}$ with matrix $I_{n_{i,i}}$ in the basis $B^0_{i,i}$, and to $V^0_{\{ i,j \}}$ 
with matrix $$\left( \begin{array}{cc}  I_{n_{i,j}} &  -\delta I_{n_{i,j}}\\
  I_{n_{i,j}} & \delta I_{n_{i,j}} \end{array} \right)$$ in the basis $B^0_{\{ i,j \}}$.  
\end{prop}

\begin{df}
 We denote by $R(P\backslash G/H)$ the set of representatives described in Proposition \ref{PH}.
\end{df}

Each element in this set of representatives has the following property, which is immediate to check:

\begin{prop}\label{admissible}
If the element $u$ of $R(P\backslash G/H)$ corresponds to $s=(n_{i,j})_{1\leq i \leq j \leq t}$, we write the set $\{ 1,\dots,n\}$ as the ordered disjoint union of 
intervals $I_{1,1} \cup I_{1,2} \cup \dots \cup I_{t,t-1}\cup I_{t,t}$, with $I_{i,j}$ of cardinality $n_{i,j}$. Then the element $w=uu^{-\sigma}$ is the involution 
of $\mathfrak{S}_n$ which fixes the intervals $I_{i,i}$, and exchanges the intervals $I_{i,j}$ and $I_{j,i}$ for different $i$ and $j$.

Moreover $u$ is $P_s$-admissible, i.e. the Levi subgroup $M_s$ is normalized by $w$. 
\end{prop}

\begin{rem}\label{JLR}
To see that we obtain the same set as in \cite{JLR}, according to Proposition 1.1.3 of \cite{C}, it is enough to check the inclusion $w(\Phi_M^+)\subset \Phi^+$. 
For $\alpha$ in $\Phi_M^+$, two cases occur: either $\alpha$ belongs to $\Phi_s^+$, i.e. $Lie(N_\alpha)$ lies in a diagonal block $n_{i,j}\times n_{i,j}$, 
which is exchanged by $w$ with the diagonal block $n_{j,i}\times n_{j,i}$ without changing the coefficients inside the block, hence $w(\alpha)$ belongs to $\Phi_s^+$, 
otherwise $Lie(N_\alpha)$ lies in a block $n_{i,j}\times n_{i,k}$ with $j<k$, which is sent by $w$ to the sub-block $n_{j,i}\times n_{k,i}$ of the block 
$n_{j}\times n_{k}$, hence $w(\alpha)$ belongs to $\Phi^+ -\Phi_M^+$ in this case.\\
Another consequence of the preceding discussion is that $M_s$ is actually $M \cap M^w$ (so this is how $s$ would appear in \cite{JLR}).
\end{rem} 

\section{Structure of the group $P(K)\cap uG_n(F)u^{-1}$ and modulus characters}\label{structure}

Let $u$ be an element of $R(P\backslash G/H)$, corresponding to a sequence $s=(n_{i,j})_{1\leq i \leq j \leq t}$ in $I(\bar{n})$. We want to analyze the structure of the group 
$P\cap uHu^{-1}$. If $u$ is $P$-admissible (which is if $s$ equals $(n_1,\dots,n_t)$), it is done in \cite{JLR}.
 However, the group $P\cap uHu^{-1}$ is actually equal to $P_s\cap uHu^{-1}$, we then refer to \cite{JLR}, as $u$ is $P_s$-admissible.

\begin{prop}
 Let $u$ be an element of $R(P\backslash G/P)$, corresponding to a sequence $s=(n_{i,j})_{1\leq i \leq j \leq t}$ in $I(\bar{n})$, 
the group $P\cap uHu^{-1}$ is equal to $P_s\cap uHu^{-1}$.
\end{prop}

\begin{proof}
Denoting by $w$ the element $uu^{-\sigma}$, the group $uHu^{-1}$ is the subgroup of $G$ fixed by the involution $\theta: x \mapsto w^{-1} x^\sigma w$, we denote it by
$G^{<\theta>}$, and we more generally denote by $S ^{<\theta>}$ the fixed points of $\theta$ in $S$, for any subset $S$ of $G$. Hence we need to show the equality
$P^{<\theta>}= P_s^{<\theta>}$, which is equivalent to the inclusion $P^{<\theta>} \subset P_s$. We actually show the stronger inclusion 
$P \cap P^w \subset P_s$ (here $P^w$ denotes $w^{-1}Pw$ which is equal to $\theta(P)$ as $\sigma$ fixes $P$). As we are dealing with $GL(n)$, it is enough to show the inclusion
 of Lie algebras 
$Lie(P) \cap Lie(P)^w \subset Lie(P_s)$ as the groups in question are the invertible elements of their Lie algebra. But then 
$Lie(P)=Lie(N_s^-) \oplus Lie(P_s)$ and $Lie(P) \cap Lie(P)^w= Lie(N_s^-) \cap Lie(P)^w \oplus Lie(P_s) \cap  Lie(P)^w$, because as $w$ belongs to $W$, all of 
these algebras decompose as a direct sum of subspaces $Lie(N_{\alpha})$ for some $\alpha$ in $\Phi$ and maybe of the Lie algebra $Lie(T)$.\\ 
We show that $Lie(N_s^-) \cap Lie(P)^w$ is zero by showing that $Lie(N_s^-)^w \cap Lie(P)$ is zero, as $w$ is an involution.
But if $\alpha$ is such that we have
$Lie(N_{\alpha}) \subset Lie(N_s^-)$, then $\alpha$ belongs to $\Phi_M^- -\Phi_s^-$. As $\Phi_s=\Phi_{M\cap M^w}$, we must have $w(\alpha)$ outside $\Phi_{M}$,
 but because of Remark \ref{JLR}, we know that $w(\Phi_M^-) \subset \Phi^-$, hence $w(\alpha)$ lies in $\Phi^- {-\Phi_M^-}$.

\end{proof}

Now Lemma 21 of \cite{JLR} applies, and we deduce:  

\begin{prop}\label{decomposition}
 Let $u$ be an element of $R(P\backslash G/H)$, corresponding to a sequence $s=(n_{i,j})_{1\leq i \leq j \leq t}$ in $I(\bar{n})$, and let $w$ be the involution 
$uu^{-\sigma}$ of the Weyl group of $G$. We denote by $\theta$ the involution $x \mapsto w^{-1} x^\sigma w$ of $G$. 
 The group $P^{<\theta>}= P\cap uHu^{-1}$ is the semi-direct
 product of $M_s^{<\theta>}=M_s \cap uHu^{-1}$ and $N_s^{<\theta>}=N_s \cap uHu^{-1}$.\\
The group $M_s^{<\theta>}$ is given by the matrices 
$$\begin{bmatrix} 
A_{1,1} &        &        &       &       &        &        &       &        \\
        & \ddots &        &       &       &        &        &       &        \\
        &        & A_{1,t}&       &       &        &        &       &        \\
        &        &        &A_{2,1}&       &        &        &       &        \\
        &        &        &       &\ddots &        &        &       &        \\
        &        &        &       &       & A_{2,t}&        &       &        \\
        &        &        &       &       &        & A_{t,1}&       &        \\
        &        &        &       &       &        &        &\ddots &        \\
        &        &        &       &       &        &        &       &  A_{t,t} 
  \end{bmatrix}$$ with $A_{j,i}=A_{i,j}^{\sigma}$ in $M_{n_{i,j}}(K)$.\\ 

\end{prop}

We will use the following fact later.

\begin{prop}\label{util}

If we denote by $P'_s$ the standard parabolic subgroup of $M$ associated with the partition of $n$ corresponding to the sequence $s=(n_{i,j})$ in $I(\bar{n})$, 
by $u$ the element of $R(P\backslash G/H)$ corresponding to $s$, and by $N'_s$ the unipotent radical of $P'_s$, then the following inclusion is true:
$$N'_s \subset N_s^{<\theta>}N.$$
\end{prop}

\begin{proof} Let $x$ be an element of an elementary unipotent subgroup $N_{\alpha}$ of $N'_s$. We already saw the inclusion 
$w(\Phi_M^+ -\Phi_s^+)\subset \Phi^+ -\Phi_M^+$, which implies 
$\theta(N'_s) \subset N$, 
so $\theta(x)$ belongs to $N$. The elements of $N_\alpha$ and $N_w(\alpha)$ commute.
If it wasn't the case, then $\alpha+w(\alpha)$ would be a (positive) root. 
But $w$ is an involution in $\mathfrak{S}_n$, writing $\alpha$ as $e_i - e_j$ for $i<j$, so that $w(\alpha)= e_{w(i)} - e_{w(j)}$, if
$\alpha+w(\alpha)$ was a root, one would have $w(i)=j$, i.e. $w$ would exchange $i$ and $j$, and $w(\alpha)$ wouldn't be positive.\\
The consequence of this is that $x\theta(x)$ is fixed by $\theta$, which implies that $x$ belongs to $N_s^{<\theta>}N$. 
Now suppose we know that all $y$ in $N'_s$ of the form
 $x_n \dots x_1$, for $x_i$ in an elementary unipotent subgroup of $N'_s$, belong to $N_s^{<\theta>}N$. Let $z$ be of the form $z=x_{n+1}y$ for $y$ of the form described
 just before, and $x_{n+1}$ in an elementary unipotent subgroup $N_{\alpha}$ of $N'_s$. Then by hypothesis, there is $y'$ in $N$ such that $h=yy'$ belongs to 
$N_s^{<\theta>}$. But $\theta(x_{n+1})$ belongs to $N$, and as $h$  belongs to $N_s^{<\theta>}\subset P$, the element $h^{-1}\theta(x_{n+1})h$ is in $N$. 
Finally $zy'h^{-1}\theta(x_{n+1})h= x_{n+1}\theta(x_{n+1})h$ belongs to $N_s^{<\theta>}$, and $y'h^{-1}\theta(x_{n+1})h$ belongs to $N$, so $z$ belongs to 
$N_s^{<\theta>}N$. This concludes the proof as any element of $N'_s$ is a product of elements in root subgroups of $N'_s$.
\end{proof}

\begin{prop}\label{modulus}
One has the following equality: $$(\delta_P \delta_{ P'_s})_{|M_s^{<\theta>}}= {\delta^2_{P^{<\theta>}}}_{|M_s^{<\theta>}}.$$
\end{prop}

\begin{proof} 
The group $M_s^{<\theta>}$ is the $F$-points of a reductive group (as it is the fixed points of $M_s$ under an involution defined over $F$).
 As this is an equality of positive characters, by Lemma 1.10 of 
\cite{KT}, it is enough to check the equality on the ($F$-points of the) $F$-split component $Z_s^{<\theta>}$ (the maximal $F$-split torus in the center
of $M_s$) of $M_s^{<\theta>}$. In our case, the group $Z_s^{<\theta>}$ is given by the matrices $$\begin{bmatrix} 
\lambda_{1,1} I_{1,1}&        &        &       &       &        &        &       &        \\
        & \ddots &        &       &       &        &        &       &        \\
        &        & \lambda_{1,t} I_{1,t}&       &       &        &        &       &        \\
        &        &        &\lambda_{2,1} I_{1,1}&       &        &        &       &        \\
        &        &        &       &\ddots &        &        &       &        \\
        &        &        &       &       & \lambda_{2,t} I_{2,t}&        &       &        \\
        &        &        &       &       &        & \lambda_{t,1} I_{t,1}&       &        \\
        &        &        &       &       &        &        &\ddots &        \\
        &        &        &       &       &        &        &       & \lambda_{t,t} I_{t,t} 
  \end{bmatrix}$$ with $\lambda_{i,j}=\lambda_{j,i}$ in $F^*$.\\ 
The characters $\delta_P$ and $\delta_{P'_s}$ verify $\delta_P(x)=|det(Ad(x)_{|Lie(N)})|_K$ and $\delta_{P'_s}(x)=|det(Ad(x)_{|Lie(N'_s)})|_K$,
 hence the equality $\delta_P \delta_{P'_s}= \delta_{P_s}$ holds as we have 
$Lie(N_s)=Lie(N'_s)\oplus Lie(N)$. But $Z_s^{<\theta>}$ is a subgroup of $P_s(F)$, and one has the relation $\delta_{P_s}(t)=\delta^2_{P_s(F)}(t)$ for 
$t$ in $P_n(F)$, because $|\ |_K$ restricts as $|\ |^2_F$ to $F$. 
So finally we only need to prove the equality $\delta_{P_s}(F)= \delta_{P^{<\theta>}}=  \delta_{P_s^{<\theta>}}$ on $Z_s^{<\theta>}$.\\
We denote by $\mathfrak{N}_{\alpha,w(\alpha)}$ the $F$-vector space $\{x \in Lie(N_\alpha)+ Lie(N_{w(\alpha)}): \ \theta(x)=x \}$ of dimension 
$|\{\alpha,w(\alpha)\}|$, so that 
$Lie(N_s^{<\theta>})$ is the direct sum of the subspaces $\mathfrak{N}_{\alpha,w(\alpha)}$ for $\{\alpha,w(\alpha)\} \subset \Phi^+ -\Phi_s^+$.\\  
Let $t$ be in $Z_s^{<\theta>}$, one has:
$$ \delta_{P_s^{<\theta>}}(t)= \underset{\{ \alpha,w(\alpha) \} \subset \Phi^+ -\Phi_s^+}{\prod} |det(Ad(t)_{|\mathfrak{N}_{\alpha,w(\alpha)}})|_F=
\underset{\{\alpha \in \Phi^+ -\Phi_s^+:\ w(\alpha)\in \Phi^+ -\Phi_s^+\}}{\prod} |\alpha(t)|_F.$$

But we also have $$\underset{\{\alpha \in \Phi^+ -\Phi_s^+:\ w(\alpha)\notin \Phi^+ -\Phi_s^+\}}{\prod} |\alpha(t)|_F=
\underset{\{\alpha \in \Phi^+ -\Phi_s^+:\ w(\alpha)\in \Phi^- -\Phi_s^-\}}{\prod} |\alpha(t)|_F=1.$$

The firs equality comes from the equality $w(\Phi_s)=\Phi_s$. The second comes from the fact that as 
$t$ belongs to $Z_s^{<\theta>}$ (so that $t^w=t^\sigma=t$),
we have $|\alpha(t)|_F|-w(\alpha)(t)|_F=1$ if $\alpha\neq-w(\alpha)$, otherwise $\alpha(t)=-w(\alpha)(t)=\alpha(t)^{-1}$ implies $|\alpha(t)|_F=1$.\\

Multiplying both equalities, we finally get:
$$ \delta_{P_s^{<\theta>}}(t)=\underset{\{\alpha \in \Phi^+ -\Phi_s^+\}}{\prod} |\alpha(t)|_F=\delta_{P_s(F)}(t).$$
 
 \end{proof}

\section{Distinguished generic representations and Asai $L$-functions}

We recall that an irreducible representation $\pi$ of $G_n(K)$ is called generic if there is a non trivial character $\psi$ of $(K,+)$, such that the space of linear
 forms $\lambda$ on $V$, which verify $\lambda(\pi(n)v)=\psi(n)v$ (where by abuse of notation, we denote by $\psi(n)$ the complex number $\psi(n_{1,2}+\dots+n_{n-1,n}$)
 for $n$ in $N_n(K)$ and $v$ in $V$, is of dimension $1$.\\
 If $\pi$ is generic, the previous invariance property holds for any non trivial character $\psi$ of $K$.
A generic representation is isomorphic, up to unique (modulo scalars) isomorphism to a submodule of $Ind_{N_n(K)}^{G_n(K)}(\psi)$.
 We denote $W(\pi,\psi)$ this model of $\pi$ on which $G_n(K)$ acts by right translation, and call it the Whittaker model of $\pi$.\\
In \cite{F4}, the Asai $L$-function $L_F^K(\pi)$ of a generic representation $\pi$ is defined ``\`{a} la Rankin-Selberg'' as the gcd of a family of integrals
 of functions in $W(\pi,\psi)$ depending on a complex parameter $s$, for $\psi$ trivial on $F$. We refer to Sections 3 and 4 of \cite{M3} for a survey
 of the main properties of the Rankin-Selberg type Asai $L$-function of a generic representation.\\
The following theorem due to Zelevinsky (Th. 9.7 of \cite{Z}), classifies the generic representations of the group $G_n(K)$ in terms of quasi-square-integrable ones:

\begin{thm}\label{classgen}
Let $\bar{n}=(n_1,\dots,n_t)$ be a partition of $n$, and let $\Delta_i$ be a quasi-square-integrable representation of $G_{n_i}(K)$ for $i$ between $1$ and $t$, 
the representation $\pi=\Delta_1\times \dots \times \Delta_t$ of the group $G_n(K)$ is irreducible if and only if no $\Delta_i$'s are linked,
 in which case $\pi$ is generic. If $(m_1,\dots,m_{t'})$ is another partition of $n$, and if the $\Delta'_j$'s are unlinked segments of length
 $m_j$ for $j$ between $1$ and $t'$, then the representation $\pi$ equals $ \Delta'_1 \times \dots \times \Delta'_{t'}$ if and only if $t=t'$,
 and $\Delta_i=\Delta'_{s(i)}$ for a permutation $s$ of $\left\lbrace 1,\dots,t\right\rbrace$. Eventually, every generic representation of $G_n(K)$ is obtained this way.
\end{thm}

Now from Proposition 12 of \cite{F2}, an irreducible distinguished representation $\pi$ of the group $G_n(K)$ is Galois-autodual,
 which means that the smooth dual $\pi^{\vee}$ is isomorphic to $\pi^{\sigma}$.
A consequence of this fact and of Theorem \ref{classgen} is the following. If $\pi=\Delta_1 \times \dots \times \Delta_t$
 is a generic representation as in the statement of Theorem \ref{classgen} and if it is distinguished, then there exists a reordering
 of the ${\Delta _i}$'s, and an integer $r$ between $1$ and $t/2$, such that $\Delta_{i+1}^{\sigma} = \Delta_i^{\vee} $ for $i=1,3,..,2r-1$,
 and $\Delta_{i}^{\sigma} = \Delta_i ^{\vee}$ for $ i > 2r$. According to Theorem 6 of \cite{K}, this means that there exists a reordering of
 the ${\Delta _i}$'s, and an integer $r$ between $1$ and $t/2$, such that $\Delta_{i+1}^{\sigma} = \Delta_i^{\vee} $ for $i=1,3,..,2r-1$,
 and such that $\Delta_{i}$ is distinguished or $\eta_{K/F}$-distinguished for $i > 2r$.
 We recall that from Corollary 1.6 of \cite{AKT}, a discrete series representation cannot be distinguished, and $\eta_{K/F}$-distinguished at the same time.

\begin{thm}\label{distgen} 
Let $\pi=\Delta_1 \times \dots \times \Delta_t$ be a generic representation of the group $G_n(K)$ as in Theorem \ref{classgen}, it is 
distinguished if and only if if there is a reordering of the ${\Delta _i}$'s, and an integer $r$ between $1$ and $t/2$, such that
 $\Delta_{i+1}^{\sigma} = \Delta_i^{\vee} $ for $i=1,3,..,2r-1$, and $\Delta_{i}$ is distinguished for $i > 2r$.
\end{thm}

It is a consequence of Proposition 26 of \cite{F3}, and of the main result of \cite{M4} that representations of the form $\Delta_1 \times \dots \times \Delta_t$ 
with $\Delta_{i+1}^{\sigma} = \Delta_i^{\vee} $ for $i=1,3,..,2r-1$, and $\Delta_{i}$ distinguished for $i > 2r$, are distinguished.

Before proving the converse fact (i.e. Theorem \ref{distgen}), we recall that from the main result of \cite{M3}, this result is known to 
imply the equality of the Rankin-Selberg type Asai $L$-function $L_F^K(\pi)$ for a generic representation $\pi$ of $G_n(K)$, and of the Asai 
$L$-function $L_F^K(\rho)$ of the Langlands parameter $\rho$ of $\pi$ (see definition 2.4 of \cite{M3}). Hence the following result is also true.

\begin{thm}\label{eqasai}
 Let $\pi$ be a generic representation of the group $G_n(K)$, and let $\rho$ be the representation of dimension $n$ of the Weil-Deligne group 
$W'_K$ of $K$, corresponding to $\pi$ through Langlands correspondence. Then we have the following equality of $L$-functions:
$$ L_F^K(\pi,s) = L_F^K(\rho,s) .$$
\end{thm}

From the discussion before and after Theorem \ref{distgen}, the proof is then reduced to showing the following fact.

\begin{thm}\label{final}
Let $\pi=\Delta_1\times\dots\times\Delta_t$ be Galois autodual generic representation of the group $G_n(K)$, if it is distinguished, then there is a reordering of the ${\Delta _i}$'s, and an integer $r$ between $1$ and $t/2$, such that $\Delta_{i+1}^{\sigma} = \Delta_i^{\vee} $ for $i=1,3,..,2r-1$, and $\Delta_{i}$ is distinguished for $i > 2r$. 
\end{thm}

\begin{proof}
Let $\bar{n}=(n_1,\dots,n_t)$ be the partition of $n$ corresponding to $\pi$.
We suppose that the $\Delta_i$'s are ordered by length. Moreover as $\pi$ is Galois autodual, we suppose that inside a subsequence of same length representations,
 the first to occur are the non Galois autodual, and that at the first occurrence of such a $\Delta_{i_0}$,
 its successors are alternatively isomorphic to $\Delta_{i_0}^{\vee\sigma}$ and $\Delta_{i_0}$, until no representation among the $\Delta_i$'s is isomorphic
 to $\Delta_{i_0}$ (hence such a sub subsequence begins with a $\Delta_i$ isomorphic to $\Delta_{i_0}$, and ends with a $\Delta_i$ isomorphic to $\Delta_{i_0}^{\vee\sigma}$).\\

 Let $\Delta$ be the representation 
$\Delta_1\otimes\dots\otimes\Delta_t$ of $P$, from Lemma 4 of \cite{F4}, the $H$-module $\pi$ has a factor series with factors the representations 
$ind_{u^{-1}Pu\cap H}^H((\delta_P^{1/2}\Delta)^u)$ 
(with $(\delta_P^{1/2}\Delta)^u (x)=\delta_P^{1/2}\Delta(uxu^{-1})$) when $u$ describes $R(P\backslash G/H)$. Hence if $\pi$ is distinguished, one of 
these representations admits a nonzero $H$-invariant linear form on its space. 
This implies that there is $u$ in $R(P\backslash G/H)$ such that the representation $ind_{P\cap uHu^{-1}} ^{uHu^{-1}}(\delta_P^{1/2}\Delta)$ admits a nonzero 
$uHu^{-1}$-invariant linear form on its space. Then Frobenius reciprocity law says that $Hom_{uHu^{-1}}(ind_{P\cap uHu^{-1}} ^{uHu^{-1}}(\delta_P^{1/2}\Delta),1)$ is 
isomorphic as a vector space, to $Hom_{P\cap uHu^{-1}}(\delta_P^{1/2}\Delta,\delta_{P\cap uHu^{-1}})= Hom_{P\cap uHu^{-1}}(
 {\delta_P^{1/2}}/{\delta_{P\cap uHu^{-1}}}\Delta, 1)$.\\
 Hence there is on the space $V_{\Delta }$ of $\Delta$ a linear nonzero form $L$, such that for every $p$ in $P\cap uHu^{-1}$ and for every $v$ in $V_{\Delta }$, 
one has $L(\chi(p)\Delta(p)v)=L(v)$, where $\chi(p)=\frac{\delta_P^{1/2}}{\delta_{P\cap uHu^{-1}}}(p)$. As both $\delta_P^{1/2}$ and $\delta_{P\cap uHu^{-1}}$ are trivial 
on $N_s\cap uHu^{-1}$, so is $\chi$. Now, if $s$ is the element of $I(\bar{n})$ corresponding to $u$, let $n'$ belong to $ N'_s$, from Proposition \ref{util}, we can write $n'$ 
as a product $n_s n_0$, with $n_s$ in $N_s\cap uHu^{-1}$, and $n_0$ in $N$. As $N$ is included in $Ker(\Delta)$, one has 
$L(\Delta(n')(v))=L(\Delta(n_s n_0)(v))=L(\Delta(n_s )(v))=L(\chi(n_s)\Delta(n_s )v)= L(v)$. Hence $L$ is actually a nonzero linear form on the Jacquet module of $V_{\Delta}$ 
associated with $ N'_s$. But we also know that $L(\chi(m_s)\Delta(m_s)v)=L(v)$ for $m$ in $M_s(F)$, which reads according to Lemma \ref{modulus}: 
$L(\delta_{P'_s}^{-1/2}(m_s)\Delta(m_s)v)=L(v)$.\\
This says that the linear form $L$ is $M_s(F)$-distinguished on the normalized Jacquet module $r_{ M_s,M}(\Delta)$ (as $M_s$ is also the standard Levi subgroup associated 
with $ N'_s$).\\

The following lemma will conclude the proof of Theorem \ref{final}.

\begin{LM}
Let $\Delta_1,\ \dots,\ \Delta_t$ be unlinked segments of respectively $G_{n_1}(K),\ \dots,\ G_{n_t}(K)$, such that the set 
$\left\lbrace  \Delta_1,\dots,\Delta_t \right\rbrace $, is stable under the involution $\Delta\mapsto\Delta^{\vee\sigma}$, call $n$ the integer $n_1+\dots+n_t$, and 
$\bar{n}$ the sequence $(n_1,\dots,n_t)$. Suppose moreover that the $\Delta_i$'s are ordered by length, and that inside a subsequence of same length representations, the 
first to occur are the non Galois autodual, and that at the first occurrence of such a segment $\Delta_{i_0}$, its successors are alternatively isomorphic to 
$\Delta_{i_0}^{\vee\sigma}$ and $\Delta_{i_0}$, until no segment among the $\Delta_i$'s is isomorphic to $\Delta_{i_0}$. Then if there is $s=(n_{i,j})_{1\leq i \leq j \leq t}$ 
in $I(\bar{n})$, such that $r_{ M_{s},M}(\otimes_i \Delta_i)$ is $M_s(F)$-distinguished, there exists a reordering of the ${\Delta _i}$'s, and an integer $r$ between $1$ and
 $t/2$, such that $\Delta_{i+1}^{\sigma} = \Delta_i^{\vee} $ for $i=1,3,..,2r-1$, and $\Delta_{i}$ is distinguished for $i > 2r$. 
\end{LM}
\begin{proof}[Proof of the Lemma]
We do this by induction on $t$. It is clear for $t=1$.\\
Now suppose the result to be true for any $t'<t$.\\
If there is $s$ in $I(\bar{n})$, such that $r_{M_{s},M}(\Delta)$ is $M_s(F)$-distinguished, in particular $r_{M_{s},M}(\Delta)$ is nonzero, which from Proposition \ref{Z}, 
implies that one can write each $\Delta_i$ under the the form $[\Delta_{i,1},\dots,\Delta_{i,t}]$, with $\Delta_{i,j}$ a subsegment of $\Delta_i$ of length $n_{i,j}$. 
With these notations,  $r_{M_{s},M}(\Delta)$ is equal to the tensor product $\Delta_{1,1}\otimes\dots\otimes\Delta_{t,t}$. Hence the fact that $r_{M_{s},M}(\Delta)$ is 
distinguished by the group $M_s(F)$ is equivalent to the fact that $\Delta_{i,i}$ is $G_{n_{i,i}}(F)$-distinguished if $n_{i,i}\neq 0$, and 
$\Delta_{j,i}=\Delta_{i,j}^{\vee\sigma}$ if $i<j$ and $n_{i,j}\neq 0$.\\
Let $i_0$ be the smallest $i$, such that $\Delta_{1,i}$ (or equivalently $n_{1,i}$) is nonzero.
\begin{enumerate}
 \item $\mathbf{i_0=1}$: the representation $\Delta_{1,1}$ is distinguished, hence Galois autodual. If $\Delta_1= [\Delta_{1,1},\dots,\Delta_{1,t}]$ was not equal to 
$\Delta_{1,1}$, then one would have $\Delta_1^{\vee\sigma}= [\Delta_{1,t}^{\vee\sigma},\dots,\Delta_{1,1}]\neq \Delta_1$. But the segment $\Delta_1^{\vee\sigma}$ would also
 occur, and $\Delta_1$ and $\Delta_1^{\vee\sigma}$ would be linked, which is absurd. Hence $\Delta_1= \Delta_{1,1}$ is distinguished, and $n_{1,i}=0$ if $i>1$.
 We conclude by applying our induction hypothesis to the family $\Delta_2,\dots,\Delta_t$, the integer $n-n_1$ with partition $(n_2,\dots,n_t)$, and sub-partition 
$s'=(n_{i,j}| \ i\geq 2,j\geq 2)$.\\

\item $\mathbf{i_0>1}$: one has $\Delta_{i_0,1}=\Delta_{1,i_0}^{\vee\sigma}$. As the representation $\Delta_{i_0}$ is either Galois autodual, or coupled with 
$\Delta_{i_0}^{\vee\sigma}$, the representation $\Delta_{i_0}^{\vee\sigma}=[\Delta_{i_0,t}^{\vee\sigma},\dots,\Delta_{1,i_0}]$ occurs. But because the representation 
$\Delta_1$ has the smallest length among the $\Delta_i$'s, the segments $\Delta_1$ and  $\Delta_{i_0}^{\vee\sigma}$ would be linked unless $\Delta_1= \Delta_{1,i_0}$, which 
thus must be the case. In particular one has $n_{1,i}=0$ for $i\neq i_0$.\\ 

Two cases occur.
\begin{description}
 \item[a)] $\mathbf{\Delta_1}=\Delta_{1,i_0}$ \textbf{is Galois autodual}: if $\Delta_{i_0}$ wasn't equal to $\Delta_{i_0,1}$, then the two occurring segments 
$\Delta_{i_0}= [\Delta_{i_0,1},\dots,\Delta_{i_0,t}]$ and $\Delta_{i_0}^{\vee\sigma}=[\Delta_{i_0,t}^{\vee\sigma},\dots,\Delta_{i_0,1}]$ would be linked, 
and that is not the case. Hence we have $\Delta_{i_0}=\Delta_{i_0,1}=\Delta_1$, and $n_{i_0,j}=0$ for $j\neq 1$.  We conclude by applying our induction hypothesis 
to the family $\Delta_2,\dots,\Delta_{i_0-1},\Delta_{i_0+1},\dots,\Delta_t$, the integer $n-n_1-n_{i_0}$ with partition $(n_2,\dots,n_{i_0-1},n_{i_0+1 },\dots,n_t)$, 
and sub partition $s'=(n_{i,j}| \ i\neq i_0, \ j \neq 1)$.\\
  
\item[b)] $\mathbf{\Delta_1}=\Delta_{1,i_0}$ \textbf{is not Galois autodual}: in this case $\Delta_2$ is $\Delta_1^{\vee\sigma}$ because of our ordering.\\
Let $j_0$ be the smallest $j$, such that $\Delta_{2,j}$ (or equivalently $n_{2,j}$) is nonzero. If $j_0=2$, as in the case $\mathbf{i_0=1}$, one has $\Delta_2=\Delta_{2,2}$, 
and we conclude by applying our induction hypothesis to the family $\Delta_1,\Delta_3,\dots,\Delta_t$, the integer $n-n_2$ with partition $(n_1,n_3,\dots,n_t)$, and sub 
partition $s'=(n_{i,j}|\  i\neq 2,j\neq 2)$.\\
If $j_0\neq 2$, then $\Delta_2$ must be equal to $\Delta_{2,j_0}$.
It is indeed clear for $j_0>2$, otherwise $\Delta_2$ and $\Delta_{j_0}^{\vee\sigma}$ would be linked, and in the case $j_0=1$,
 then one has $i_0=2$, and $\Delta_2$ is equal to $[\Delta_{2,1},\dots,\Delta_{2,t}]$ but also 
to $\Delta_1^{\vee\sigma}=[\dots,\Delta_{2,1}]$, so that $\Delta_2$ is $\Delta_{2,1}$.\\
This implies $n_{2,j}=0$ for $j\neq j_0$. Thus we  have $\Delta_2=\Delta_{2,j_0}= \Delta_{1,i_0}^{\vee\sigma}=\Delta_1^{\vee\sigma}$. 
But the two occurring segments $\Delta_{i_0}^{\vee\sigma}=[\Delta_{i_0,t}^{\vee\sigma},\dots,\Delta_1]$ and $\Delta_{j_0}=[\Delta_{j_0,2}=\Delta_1,\dots,\Delta_{j_0,t}]$ 
will be linked unless either $\Delta_{i_0}=\Delta_1^{\vee\sigma}$, in which case we conclude by applying our induction hypothesis to the family 
$\Delta_2,\dots,\Delta_{i_0-1},\Delta_{i_0+1},\dots,\Delta_t$, the integer $n-n_1-n_{i_0}$ with partition $(n_2,\dots,n_{i_0-1},n_{i_0+1 },\dots,n_t)$,
 and sub-partition $s'=(n_{i,j}|\ i\neq i_0, \ j \neq 1)$, or $\Delta_{j_0}=\Delta_1$, in which case we conclude by applying our induction hypothesis
 to the family $(\Delta_j|\ j\neq 2 \ and \ j_0)$, the integer $n-n_2-n_{j_0}$ with partition $(n_j|\ j\neq 2 \ and \ j_0)$, and sub partition 
$s'=(n_{i,j}|\ i \neq 2, j \neq j_0)$.
 
\end{description}

\end{enumerate}

 \end{proof}
\end{proof}

\end{document}